\documentclass[12pt]{amsart}
\usepackage[margin=1in]{geometry}
\usepackage{amsmath, amssymb, amsthm,hyperref}
\usepackage{mathtools,xcolor}
\usepackage{enumitem}
\usepackage{multicol}

\newtheorem{theorem}{Theorem}[section]
\newtheorem{lemma}[theorem]{Lemma}
\newtheorem{proposition}[theorem]{Proposition}
\newtheorem{corollary}[theorem]{Corollary}
\theoremstyle{definition}

\newtheorem{example}[theorem]{Example}

\newtheorem{remark}[theorem]{Remark}

\newcommand{\g}{\mathfrak{g}}

\newcommand{\h}{\mathfrak{h}}

\numberwithin{equation}{section} 
\usepackage{color}

\begin{document}

\title[Bounds for asymptotic characters of simple Lie groups]{Bounds for asymptotic characters of simple Lie groups}

\author{Pavel Etingof}

\address{Department of Mathematics, MIT, Cambridge, MA 02139, USA}

\author{Eric Rains}

\address{Department of Mathematics, California Institute of Technology, Pasadena,
CA 91125, USA}

\maketitle

\tableofcontents

\centerline{\bf To Tom Koornwinder on his 80-th birthday with admiration}  

\begin{abstract} An important function attached to a complex simple Lie group $G$
is its \linebreak asymptotic character $X(\lambda,x)$ (where $\lambda,x$ are real (co)weights of $G$) -- the Fourier transform in $x$ of its Duistermaat-Heckman function $DH_\lambda(p)$ (continuous limit of weight multiplicities). It is shown in \cite{GGR} that the best $\lambda$-independent upper bound $-c(G)$ for ${\rm inf}_x{\rm Re}X(\lambda,x)$ for fixed $\lambda$  is strictly negative. We quantify this result by providing a lower bound for $c(G)$ in terms of $\dim G$. We also provide upper and lower bounds for $DH_\lambda(0)$ when $|\lambda|=1$. This allows us to show that $|X(\lambda,x)|\le C(G)|\lambda|^{-1}|x|^{-1}$ for some constant $C(G)$ depending only on $G$, which implies the conjecture 
in Remark 17.16 of \cite{GGR}. We also show that $c(SL_n)\le (\frac{4}{\pi^2})^{n-2}$. 
Finally, in the appendix we prove Conjecture 1 in \cite{CZ} about Mittag-Leffler type sums for $G$. 
\end{abstract}

\section{Introduction} 

\subsection{Asymptotic characters of simple Lie groups and Duistermaat-Heckman measures}
Let $G$ be a simply connected simple complex Lie group, $G_c\subset G$ a maximal compact subgroup, $\g_c:={\rm Lie}G_c$ its Lie algebra, $\h_c\subset \g_c$ a Cartan subalgebra, and $\g={\rm Lie}G,\h$ their complexifications. Let $\h_{\Bbb R}$ be the real part of $\h$, so $\h_c=i\h_{\Bbb R}$; then $\h_{\Bbb R}$ carries a Euclidean norm induced by the Killing form on $\g$ and $\h_{\Bbb R}^*$ carries the dual norm. Let $\omega_j\in \h^*$ be the fundamental weights of $G$. Given a dominant weight $\lambda=\sum_j \lambda_j\omega_j$, $\lambda_j\in \Bbb R_{\ge 0}$, let $[\lambda]:=\sum_j [\lambda_j]\omega_j$, where $[a]$ is the floor of $a$. Let $P_+$ be the set of dominant integral weights for $G$. For $\mu\in P_+$ let $L_\mu$ be the irreducible $G$-representation with highest weight $\mu$ and $\chi_\mu: G\to \Bbb C$ be its character. It is well known that for all $x\in \g$ there exists a limit 
$$
X(\lambda,x):=\lim_{N\to \infty} \frac{\chi_{[N\lambda]}(e^{\frac{ix}{N}})}{\chi_{[N\lambda]}(1)}
$$
called the {\bf asymptotic character} of $G$ (\cite{He}); 
namely, using Kirillov's character formula (\cite{Ki}), 
$$
X(\lambda,x)=\int_{O_\lambda}e^{i(b,x)}db,
$$ 
where $O_\lambda\subset \g^*$ is the $G_c$-orbit of $\lambda$ and $db$ is the invariant probability measure on $O_\lambda$ and $(,)$ denotes the natural pairing between $\mathfrak h$ and $\mathfrak h^*$. In other words, $X(\lambda,x)$ is the Fourier transform of the delta-distribution of the orbit $O_\lambda$ (analytically continued to the complex domain). 
In particular, $X(\lambda,x)$ extends to an entire function in $\lambda\in \h^*, x\in \h$ (given by the same formula) such that $X(\lambda,0)=X(0,x)=1$ and $|X(\lambda,x)|\le 1$ 
for $\lambda\in \h_{\Bbb R}^*$, $x\in \h_{\Bbb R}$.

Let $R_+\subset \h^*$  be the set of positive roots, $\alpha^\vee$ the coroot corresponding to a root $\alpha\in R_+$, 
$$
\delta:=\prod_{\alpha\in R_+}\alpha,\ \delta_*:=\prod_{\alpha\in R_+}\alpha^\vee,\ \rho:=\frac{1}{2}\sum_{\alpha\in R_+}\alpha,
$$ 
and $W$ be the Weyl group of $G$. Taking a limit in the Weyl character formula, we get 
$$
X(\lambda,x)=i^{-|R_+|}\delta_*(\rho)\frac{\sum_{w\in W} \det(w)e^{i(\lambda,wx)}}{\delta(x)\delta_*(\lambda)}.
$$
This formula was obtained by Harish-Chandra in 1957 (\cite{HC}). 

It also follows that for $\lambda\in \h_{\Bbb R}^*\setminus 0$ the restriction of $X(\lambda,x)$ to $x\in \h_{\Bbb R}$ is the Fourier transform of the {\bf Duistermaat-Heckman measure} $DH_\lambda=\pi_* db$, where $\pi: O_\lambda\to \h_{\Bbb R}^*$ is the natural projection (\cite{GLS,He}). This measure has the form $DH_\lambda=DH_\lambda(p)dp$, where $DH_\lambda(p)$ is an integrable (in fact, piecewise polynomial) function called the {\bf Duistermaat-Heckman function}. 
By Riemann's lemma, it therefore follows that if $\lambda\ne 0$ then 
for any $x\ne 0$, 
\begin{equation}\label{limi}
\lim_{t\to +\infty}X(\lambda,tx)=0.
\end{equation}
For more information on the asymptotic behavior of the ratio $\frac{\chi_\lambda(e^{ix})}{\chi_\lambda(1)}$ and on the function $X(\lambda,x)$, see \cite{He,GGR}.

\subsection{A lower bound for minimax of minus the real part of the asymptotic character} 
For $\lambda\in \h_{\Bbb R}^*\setminus 0$ let 
$$
c(G,\lambda):=-\inf_{x\in \h_{\Bbb R}} {\rm Re}X(\lambda,x)
$$ 
(so $c(G,\lambda)\ge 0$ by \eqref{limi})
and 
$$
c(G)=\inf_{\lambda\in \h^*_{\Bbb R}\setminus 0} c(G,\lambda)=
-\sup_{\lambda\in \h^*_{\Bbb R}\setminus 0} \inf_{x\in \h_{\Bbb R}}{\rm Re}X(\lambda,x)=-\sup_{\lambda\in \h_{\Bbb R}^*: |\lambda|=1} \inf_{x\in \h_{\Bbb R}}{\rm Re}X(\lambda,x)
$$
(the last equality holds because $X(t\lambda,x)=X(\lambda,tx)$ for $t\in \Bbb R$).
Thus for any $\lambda\in \h^*_{\Bbb R}\setminus 0$
$$
\inf_{x\in \h_{\Bbb R}}{\rm Re}X(\lambda,x)\le -c(G).
$$ 
One of the main results of \cite{GGR}\footnote{Note that the definition of $X(\lambda,x)$ in \cite{GGR} differs from ours by multiplication of $x$ by $i$.} is that for any $G$, one has $c(G)>0$.  In particular, $c(G,\lambda)>0$ for all $\lambda$. 

For example, for $G=SL_2$ we have 
$$
X(\lambda,x)=\frac{\sin \lambda x}{\lambda x}, 
$$
so
$$
c(SL_2)=-\min_{\theta \in \Bbb R}\frac{\sin \theta}{\theta}\approx 0.2172.
$$
A more complicated example is $G=Sp_4$. In this case $W$ is the group of symmetries of the square, and using the standard Cartesian coordinates in $\mathfrak{h}_{\Bbb R}\cong \mathfrak h^*_{\Bbb R}=\Bbb R^2$, we have: 
\begin{equation}\label{sp4}
X(\lambda,x)=6\frac{\sin(\lambda_1x_2)\sin(\lambda_2x_1)-\sin(\lambda_1x_1)\sin(\lambda_2x_2)}{\lambda_1\lambda_2(\lambda_1^2-\lambda_2^2)x_1x_2(x_1^2-x_2^2)}.
\end{equation} 
Numerical calculations show that 
$$
c(Sp_4)=-\min_{\theta\in \Bbb R}  3(\sqrt{2}+1)^2\frac{\sin((\sqrt{2}-1)\theta)-(\sqrt{2}-1)\sin\theta}{\theta^3}
\approx 0.0204.
$$
More precisely, this is minus the minimal value of $X(\sqrt{2}-1,1,x_1,x_2)$ 
which is attained at $x_1\approx 8.2517..$, $x_2=0$, as well as the images of this point under rotations by $\frac{\pi m}{4}$, $m\in \Bbb Z/8$. Note that 
these are {\bf not} critical points of $X(\lambda,x)$: for $b>\sqrt{2}-1$ the minimum 
of $X(b,1,x_1,x_2)$ is attained at the lines $x_1=0$, $x_2=0$, while for $b<\sqrt{2}-1$ it is attained at the lines $x_2=\pm x_1$, and we are taking the maximum over these two regimes. 

In higher rank this pattern continues and gets even trickier: the maximum is taken over many regimes whose combinatorics is fairly complicated. This makes computing and even estimating $c(G)$ in higher rank pretty difficult. 

 \subsection{Main results}
One of the goals of this note is to prove the 
following explicit lower bound for $c(G)$.  

\begin{theorem}\label{maint} Let $d:=\dim G$. Then 
$$
c(G)\ge \frac{e^{-1}(1-e^{-1})^{\frac{d}{2}+1}}{(\frac{d}{2}+1)^{\frac{d}{2}+1}(\log(\frac{d}{2}+1))^{\frac{d}{2}}}.
$$
\end{theorem} 

Theorem \ref{maint} is proved in Section 2. The proof is obtained by quantifying the technique of \cite{GGR}. 

For $d=3$ ($G=SL_2$) the bound of Theorem \ref{maint} gives about $0.01$, which is much smaller than the actual value $0.2172...$, and it goes to zero extremely fast as $d\to \infty$. So it is likely that even the leading factor of this bound, $d^{-\frac{d}{2}}$, is not sharp. However, we have not been able to improve it even to $\Omega(d^{-(1-\varepsilon)\frac{d}{2}})$ for any $\varepsilon>0$.  

We note however that, as noted in \cite{GGR}, the constant $c(SL_n)$
{\bf does} decay at least exponentially fast as $n\to \infty$, namely, $c(SL_n)\le (\frac{4}{\pi^2})^{n-2}$. 
This is shown in Section 6. Thus the decay of $c(SL_n)$
is somewhere between $(\frac{4}{\pi^2})^{n-2}$ and $\approx n^{-n^2}$. 
It would be interesting to estimate this rate more tightly. 

We also give an explicit uniform upper bound for the Duistermaat-Heckman function. 
This allows us to give a uniform upper bound for the asymptotic character.
Namely, we prove

\begin{theorem}\label{upbou} For any $G$ there exists $C(G)>0$ such that for all nonzero $x\in \h_{\Bbb R},\ \lambda\in \h^*_{\Bbb R}$
one has 
$$
|X(\lambda,x)|\le \frac{C(G)}{|\lambda|\cdot |x|}.
$$ 
\end{theorem} 

In fact, we provide a formula for $C(G)$ (formula \eqref{formcg}). 

Also, Theorem \ref{upbou} immediately implies

\begin{corollary}\label{coro1} For any $\lambda\in \mathfrak h^*_{\Bbb R}$ 
there exists $x_0=x_0(\lambda)\in \h_{\Bbb R}$ such that 
$|x_0|\le \frac{C(G)}{c(G)|\lambda|}$ and $c(G,\lambda)=-{\rm Re}X(\lambda,x_0)$.     
\end{corollary}

This allows us to prove the following corollary conjectured in \cite{GGR} (see \cite{GGR}, Remark 17.16). 

\begin{corollary}\label{GGRconj} For almost all dominant integral weights $\lambda$, 
namely at least when\footnote{By the strange formula 
of Freudental-de Vries, for the dual Killing form one has $|\rho|^2=\frac{d}{24}$, so 
this bound also equals $\frac{C(G)}{c(G)4\sqrt{3}}$.}
$$
|\lambda+\rho|>\frac{C(G)|\rho|}{c(G)\sqrt{2d}},
$$ 
there exists $x\in \h_{\Bbb R}$ such that 
$$
{\rm Re}\frac{\chi_\lambda(e^{ix})}{\chi_\lambda(1)}\le -c(G).
$$
\end{corollary} 

Finally, in the appendix we apply the techniques of the theory of Duistermaat-Heckman measures and results of P. Littelmann to prove a generalization of a conjecture of Coquereaux and Zuber (\cite{CZ}). This appendix previously appeared as a preprint \cite{ER} but was never published. We include it here since it is thematically fairly close to the subject of this paper. 
We note that since the techniques and ideas we use are fairly well known, this appendix should be viewed as largely expository.  

The paper is organized as follows. Theorem \ref{maint} (the lower bound for $c(G)$) is proved in Section 2. The upper bound for the Duistermaat-Heckman function is proved in Section 3. Theorem \ref{upbou} (the upper bound for an asymptotic character) is proved in Section 4. Corollary \ref{GGRconj} (a conjecture from \cite{GGR}) is proved in Section 5. In Section 6 we give an upper bound for $c(G)$ in the case of $G=SL_n$ (and one can use the same method to obtain similar bounds for other classical groups). Finally, in the appendix (Section 7) we consider Mittag-Leffler-type sums for $G$ and prove Conjecture 1 in \cite{CZ}. 

\begin{remark} In the paper \cite{BBO} the theory of asymptotic characters and 
Duistermaat-Heckman measures is extended to non-crystallographic finite Coxeter groups, where ordinary characters don't make sense. We expect that our bounds can be extended to this setting. 
\end{remark} 

\begin{remark}  It would be interesting to generalize our estimates to the {\it Dunkl kernel}, which is a deformation of the asymptotic character parametrized by a complex function $k$ of conjugacy classes of reflections in $W$ specializing to this character when $k=1$ (\cite{Ro}). As in \cite{GGR}, this would imply uniform bounds on the minima of real parts of Heckman-Opdam multivariate Jacobi polynomials (which by \cite{RV} converge to the Dunkl kernel in the appropriate limit), and in particular for spherical functions on compact symmetric spaces.  Note that for any Coxeter group, the Dunkl kernel satisfies an analogue of \cite[Cor. 18.12]{GGR} (the $SL_2(\Bbb Z)$-property of the Dunkl transform, see \cite[Lemma 4.13]{Ro}), and thus the argument of \cite[Lemma 18.3]{GGR} carries over to say that for $k\ge 0$ the Dunkl kernel must take some values with strictly negative real parts.  (The argument of \cite[Theorem 18.2]{GGR} also applies here: the Dunkl kernel satisfies a scaling symmetry that lets one view it as a scaled limit of itself and thus get a uniform bound by compactness.) 
\end{remark} 

\subsection{Acknowledgements} This material is based upon work supported by the National Science Foundation under Grant No. DMS-1928930 and by the Alfred P. Sloan Foundation under grant G-2021-16778, while the authors were in residence at the Simons Laufer Mathematical Sciences Institute (formerly MSRI) in Berkeley, California, during the Spring 2024 semester. We are grateful to Xuhua He for discussions related to Subsection \ref{true}, to Antoine de Saint Germain for sharing Proposition \ref{SGer}, to S. Garibaldi for reading the preliminary version of the paper, and to two anonymous referees for useful suggestions. The work of P. E. was also partially supported by the NSF grant DMS -- 2001318. P.E. is grateful to J.-B. Zuber for turning his attention to Conjecture 1 of \cite{CZ}.

\section{The uniform lower bound for the real part of the asymptotic character}

\subsection{A lower bound for ${\rm Re}\frac{\chi_\lambda(e^{x})}{\chi_\lambda(1)}$}

\begin{lemma}\label{l1} For $x\in \g_c$
$$
{\rm Re}\frac{\chi_\lambda(e^{x})}{\chi_\lambda(1)}\ge 1-\frac{(\lambda,\lambda+2\rho)|x|^2}{2d}.
$$
In particular, if $0<L<1$ and $|x|\le \sqrt{\frac{2Ld}{(\lambda,\lambda+2\rho)}}$
then 
$$
{\rm Re}\frac{\chi_\lambda(e^{x})}{\chi_\lambda(1)}\ge 1-L.
$$ 
\end{lemma}

\begin{proof} 
Let ${\rm Cas}$ be the Casimir element of $U(\g)$. 
Then ${\rm Cas}|_{L_\lambda}=(\lambda,\lambda+2\rho)$. 
For $x\in \g_c$ we have 
\begin{equation}\label{trsq}
|{\rm Tr}_{L_\lambda}(x^2)|
=\frac{|x|^2}{d}{\rm Tr}_{L_\lambda}({\rm Cas})=
\frac{(\lambda,\lambda+2\rho)|x|^2}{d}\chi_\lambda(1). 
\end{equation} 
Thus 
$$
|\partial_t^2 {\rm Re}\chi_\lambda(e^{tx})|=|{\rm Re}{\rm Tr}_{L_\lambda}(x^2e^{tx})|\le 
|{\rm Tr}_{L_\lambda}(x^2)|=\frac{(\lambda,\lambda+2\rho)|x|^2}{d}\chi_\lambda(1).
$$
Integrating this inequality twice, we obtain the lemma. 
\end{proof} 

Replacing in Lemma \ref{l1} the weight $\lambda$ with $[N\lambda]$ and $x$ with $\frac{x}{N}$, we obtain 

\begin{corollary}\label{co1} 
If $\lambda\in \h^*_{\Bbb R}\setminus 0$ and $x\in \h_{\Bbb R}$ is such that $|x|\le \frac{\sqrt{2Ld}}{|\lambda|}$ 
then 
$$
{\rm Re}X(\lambda,x)\ge 1-L.
$$ 
\end{corollary} 

\subsection{The heat kernel} 

Now consider the heat kernel $K_G(g,t)$ of the group $G_c$, which is the fundamental solution 
of the heat equation 
$$
\partial_t K_G=\Delta K_G
$$
on $G_c\times \Bbb R_{>0}$ with initial condition $K_G(g,0)=\delta_1(g)$, where $\Delta=-C$ is the Laplacian on $G_c$ and $\delta_1$ is the delta function at $1\in G_c$. 
We have 
$$
K_G(g,t)=\sum_{\mu\in P_+} e^{-t(\mu,\mu+2\rho)}\chi_\mu(g)\chi_\mu(1).
$$
Thus, integrating with respect to the Haar measure of $G_c$, we have 
\begin{equation} \label{eq2}
\int_{G_c} K_G(g,t)\tfrac{\chi_\lambda(g)}{\chi_\lambda(1)}dg=e^{-t(\lambda,\lambda+2\rho)}.
\end{equation} 
So
replacing $\lambda$ with $[N\lambda]$, $g$ with $e^{\frac{x}{N}}$, $t$ with 
$\frac{t}{N^2}$ and taking the limit $N\to \infty$, we obtain 
\begin{equation} \label{eq2a}
\int_{\g_c} K_\g(x,t)X(\lambda,-ix)dx=e^{-t|\lambda|^2},
\end{equation} 
where 
$$
K_\g(x,t)=\frac{1}{(4\pi t)^{\frac{d}{2}}}e^{-\frac{|x|^2}{4t}}
$$
is the heat kernel for $\g_c$. 

\subsection{Proof of Theorem \ref{maint} }
We assume without loss of generality that $|\lambda|=1$. 
By Corollary \ref{co1} 
$$
{\rm Re}\int_{x\in \g_c: |x|\le \sqrt{2Ld}}K_\g(x,t)
X(\lambda,-ix)dx\ge (1-L) \int_{x\in \g_c: |x|\le \sqrt{2Ld}}K_\g(x,t)dx=
$$
$$
\frac{1-L}{(4\pi t)^{\frac{d}{2}}}\int_{x\in \g_c: |x|\le \sqrt{2Ld}} e^{-\frac{|x|^2}{4t}}dx= 
\frac{S_{d-1}(1-L)}{\pi^{\frac{d}{2}}}\int_0^{\sqrt{\frac{Ld}{2t}}}r^{d-1} e^{-r^2}dr=\frac{1-L}{\Gamma(\tfrac{d}{2})}\int_0^{\frac{Ld}{2t}}s^{\frac{d}{2}-1} e^{-s}ds,
$$
where $S_{d-1}=\frac{2\pi^{\frac{d}{2}}}{\Gamma(\frac{d}{2})}$ is the area of the unit sphere in $d$ dimensions. 

Now assume that ${\rm Re}X(\lambda,y)\ge -c$ for all $y\in \h_{\Bbb R}$. For $a>0$ let 
$$
\Gamma(a,u):=\int_u^\infty s^{a-1} e^{-s}ds
$$
be the incomplete $\Gamma$-function.
We then obtain the inequality 
$$
e^{-t}\ge \frac{1-L+c}{\Gamma(\tfrac{d}{2})}\int_0^{\frac{Ld}{2t}}s^{\frac{d}{2}-1} e^{-s}ds-c=
(1-L+c)\frac{\Gamma(\frac{d}{2})-\Gamma(\tfrac{d}{2},\tfrac{Ld}{2t})}{\Gamma(\frac{d}{2})}-c.
$$
Thus 
$$
c\ge c_d(L,v),
$$ 
where 
$$
c_d(L,v)=\frac{(1-L)(\Gamma(\frac{d}{2})-\Gamma(\frac{d}{2},\frac{d}{2}v))-e^{-\frac{L}{v}}\Gamma(\frac{d}{2})}{\Gamma(\frac{d}{2},\frac{d}{2}v)}.
$$
Let us optimize this bound with respect to $L$ for a fixed $v$. 
The condition $\partial_L c_d(L,v)=0$ yields the optimal value
$$
L_0=-v\log\left(v(1-\tfrac{\Gamma(\frac{d}{2},\frac{d}{2}v)}{\Gamma(\frac{d}{2})})\right)
$$
provided that $0<L_0<1$, which means that $v<v_d$, where 
$v_d$ is the positive root of the equation 
$$
v(1-\tfrac{\Gamma(\frac{d}{2},\frac{d}{2}v)}{\Gamma(\frac{d}{2})})=1.
$$   
We have 
$$
c_d(L_0,v)=\left(1+v\log\left(v(1-\tfrac{\Gamma(\frac{d}{2},\frac{d}{2}v}{\Gamma(\frac{d}{2})})\right)-v\right)\left(\tfrac{\Gamma(\frac{d}{2})}{\Gamma(\frac{d}{2},\frac{d}{2}v)}-1\right). 
$$
This is clearly negative when $v\ge 1$, so it suffices to consider $0<v<1$.  

Now observe that
$$
\frac{\Gamma(\frac{d}{2})}{\Gamma(\frac{d}{2},\frac{d}{2}v)}-1\ge 1-\frac{\Gamma(\frac{d}{2},\frac{d}{2}v)}{\Gamma(\frac{d}{2})}=\frac{1}{\Gamma(\frac{d}{2})}\int_0^{\frac{d}{2}v} s^{\frac{d}{2}-1}e^{-s}ds\ge 
\frac{(\tfrac{d}{2}v)^{\frac{d}{2}}e^{-\frac{d}{2}v}}{\Gamma(\frac{d}{2}+1)}.
$$
Thus for $v< 1$ 
$$
c_d(L_0,v)\ge \left(1+v \log v+v \log \frac{(\tfrac{d}{2}v)^{\frac{d}{2}}e^{-\frac{d}{2}v}}{\Gamma(\frac{d}{2}+1)}-v\right)\frac{(\tfrac{d}{2}v)^{\frac{d}{2}}e^{-\frac{d}{2}v}}{\Gamma(\frac{d}{2}+1)}=
$$
$$
\left(1+(\tfrac{d}{2}+1)v \log v-v (1+\log \Gamma(\tfrac{d}{2}+1)-\tfrac{d}{2}\log (\tfrac{d}{2})+\tfrac{d}{2}v)\right)\frac{(\tfrac{d}{2}v)^{\frac{d}{2}}e^{-\frac{d}{2}v}}{\Gamma(\frac{d}{2}+1)}.
$$
Let us restrict to $v<1/3$ and also assume without loss of generality that $d\ge 8$. Then 
$$
1+\log \Gamma(\tfrac{d}{2}+1)-\tfrac{d}{2}\log (\tfrac{d}{2})+\tfrac{d}{2}v<1+\log \Gamma(\tfrac{d}{2}+1)-\tfrac{d}{2}\log (\tfrac{d}{2})+\tfrac{d}{6}<0.
$$
Hence $\frac{(\frac{d}{2})^{\frac{d}{2}}e^{-\frac{d}{6}}}{\Gamma(\frac{d}{2}+1)}\ge e$. So we see that 
$$
c_d(L_0,v)\ge \frac{(\frac{d}{2})^{\frac{d}{2}}e^{-\frac{d}{6}}}{\Gamma(\frac{d}{2}+1)}\sup_{0<v<1/3} (1+(\tfrac{d}{2}+1)v \log v)v^{\frac{d}{2}}\ge e\sup_{0<v<1/3} (1+(\tfrac{d}{2}+1)v \log v)v^{\frac{d}{2}}.
$$
Take $v=\frac{a}{(\frac{d}{2}+1)\log (\frac{d}{2}+1)}$ for $a<1$. So we obtain 
$$
c_d(L_0,v)\ge e\left(1-a+\frac{a(\log a-\log \log (\tfrac{d}{2}+1))}{\log (\frac{d}{2}+1)} \right)\left(\frac{a}{(\frac{d}{2}+1)\log (\frac{d}{2}+1)}\right)^{\frac{d}{2}}.
$$
But $a(\log a-T)\ge -e^{T-1}$, so we have 
$$
c_d(L_0,v)\ge e(1-e^{-1} -a)\left(\frac{a}{(\frac{d}{2}+1)\log (\frac{d}{2}+1)}\right)^{\frac{d}{2}}
$$
This bound is optimal for $a=\frac{\frac{d}{2}}{\frac{d}{2}+1}(1-e^{-1})$, in which case we get 
$$
c_d(L_0,v)\ge \frac{e(1-e^{-1})^{\frac{d}{2}+1}(\frac{d}{2})^{\frac{d}{2}}}{(\frac{d}{2}+1)^{\frac{d}{2}+1}}\left(\frac{\frac{d}{2}}{(\frac{d}{2}+1)^2\log (\frac{d}{2}+1)}\right)^{\frac{d}{2}}
=\frac{e(1-e^{-1})^{\frac{d}{2}+1}(\frac{d}{2})^d}{(\frac{d}{2}+1)^{3\frac{d}{2}+1}\left(\log(\frac{d}{2}+1)\right)^{\frac{d}{2}}}.
$$
Finally, note that $\left(\frac{\frac{d}{2}}{\frac{d}{2}+1}\right)^{\frac{d}{2}}>e^{-1}$. Thus we get 
$$
c(G)\ge c_d(L_0,v)\ge \frac{e^{-1}(1-e^{-1})^{\frac{d}{2}+1}}{(\frac{d}{2}+1)^{\frac{d}{2}+1}(\log(\frac{d}{2}+1))^{\frac{d}{2}}}.
$$

\section{An upper bound for the Duistermaat-Heckman function} 

Recall (\cite{GLS}) that the support $\Pi_\lambda\subset \h^*_{\Bbb R}$ of the Duistermaat-Heckman function $DH_\lambda(p)$ (the {\bf Duistermaat-Heckman polytope}) is the convex hull of the orbit $W\lambda$, i.e., the set of 
$p\in \h_{\Bbb R}^*$ with $(\lambda-wp,\omega_j^\vee)\ge 0$ for all $w\in W$ and all $j$, where $\omega_j^\vee$ are the fundamental coweights of $G$.  
Let $|\lambda|=1$ and denote by $R(\lambda)$ the largest number $R$ such that 
$|p|\le R$ implies that $p\in \Pi_\lambda$ (clearly, $R(\lambda)>0$). 
It is clear that $R(\lambda)$ is continuous in $\lambda$, 
so it attains its minimal value 
\begin{equation}\label{formrg}
R_G:=\min_{|\lambda|=1}R(\lambda)>0.
\end{equation}   

\begin{lemma}\label{twosphe} Suppose $|p|\le R_G|q|$. Then $DH_\lambda(p)\ge DH_\lambda(q)$. 
Thus for any $t>0$, $\sup_{|q|=t} DH_\lambda(q)\le \inf_{|p|\le R_Gt}DH_\lambda(p)$. 
\end{lemma} 

\begin{proof} Recall that 
$\log DH_\lambda$ is concave (\cite{O}). 
Since $p$ lies in the convex hull of $Wq$ and $DH_\lambda$ is 
$W$-invariant, we have $DH_\lambda(p)\ge DH_\lambda(q)$.
\end{proof} 

\begin{lemma}\label{averasq} (i) We have 
$$
\int_{\h_{\Bbb R}^*} DH_\lambda(p) |p|^2 dp=\frac{r}{d}=\frac{1}{{\rm h}+1},
$$
where ${\rm h}$ is the Coxeter number of $G$ and $r$ is its rank. 

(ii) For $0<L<\frac{1}{{\rm h}+1}$
$$
\int_{|p|^2\ge L}DH_\lambda(p)dp\ge \frac{1}{{\rm h}+1}-L.
$$

(iii) For $0<T<1$, 
$$
\int_{|p|^2\le \frac{1}{({\rm h}+1)T}}DH_\lambda(p)dp\ge 1-T.
$$

\end{lemma}

\begin{proof} (i)
It follows from \eqref{trsq} by passing to the limit that for every $x\in \h_{\Bbb R}$ we have 
$$
\int_{\h^*_{\Bbb R}} DH_\lambda(p) (p,x)^2 dp=\frac{|x|^2}{d}.
$$
Applying this to an orthonormal basis of $\h_{\Bbb R}$ and taking the sum, we get (i). 

(ii) Since $DH_\lambda$ is a probability measure, we have 
$$
\int_{|p|^2\le L}DH_\lambda(p) |p|^2 dp\le L. 
$$
Thus, subtracting this inequality from (i), we get 
\begin{equation}\label{Lbound}
\int_{|p|^2\ge L}DH_\lambda(p) |p|^2 dp\ge \frac{1}{{\rm h}+1}-L.
\end{equation} 
In particular, since for $p\in \Pi_\lambda$ we have $|p|\le 1$, we obtain (ii). 

(iii) By (i) 
$$
\frac{1}{{\rm h+1}}\ge \int_{|p|^2\ge \frac{1}{({\rm h}+1)T}}DH_\lambda(p)|p|^2dp\ge 
\frac{1}{({\rm h}+1)T}\int_{|p|^2\ge \frac{1}{({\rm h}+1)T}}DH_\lambda(p)dp,
$$
so 
$$
\int_{|p|^2\ge \frac{1}{({\rm h}+1)T}}DH_\lambda(p)dp\le T,
$$
which implies the statement. 
\end{proof} 

\begin{corollary} \label{supinfcor} (i) For $0<L<\frac{1}{{\rm h}+1}$ we have 
$$
\sup_{p: |p|^2=L}DH_\lambda(p)\ge \frac{\Gamma(\frac{r}{2})}{2\pi^{\frac{r}{2}}}\left(\frac{1}{{\rm h}+1}-L\right).
$$

(ii) For $0<L<\frac{1}{{\rm h}+1}$ we have 
$$
A(L):=\inf_{p: |p|^2=R_G^2L}DH_\lambda(p)\ge \frac{\Gamma(\frac{r}{2})}{2\pi^{\frac{r}{2}}}\left(\frac{1}{{\rm h}+1}-L\right).
$$
\end{corollary} 

\begin{proof} (i) This follows from the fact that for any $p\ne 0$, 
the function $t\mapsto DH_\lambda(tp)$ 
is decreasing on $[0,\infty)$ (as follows, e.g., from its log-concavity) and Lemma \ref{averasq}(ii).

(ii) This follows from (i) and Lemma \ref{twosphe}.
\end{proof} 

For $z>0$, let $b_0(z)$ be the minimum point for the function $f_z(b):=(\frac{e^z}{b})^{\frac{1}{1-b}}$. 
Thus $b_0(z)$ is the inverse function to the strictly decreasing function $b\mapsto b^{-1}+\log b-1$ mapping $(0,1)$ onto $(0,\infty)$. It is easy to see that $b_0(z)\sim \frac{1}{z}$ as $z\to \infty$, and 
$b_0(z)<\frac{1}{z+1}$. The minimal value $\min f_z$ of $f_z(b)$ then equals 
$(\frac{e^z}{b_0(z)})^{\frac{1}{1-b_0(z)}}$, hence 
we have 
$$
\min f_z\sim (ze^z)^{1+\frac{1}{z}}\sim ze^{z+1},\ z\to \infty.
$$
To make formulas more compact, define
\begin{equation}\label{formmg}
M_G:=\frac{\sqrt{{\rm h}+1}}{R_G}, 
\end{equation} 
where $R_G$ is defined by \eqref{formrg}. 
Let
\begin{equation}\label{formdg}
D_G:=\min_{b\in (0,1)} (b^{-1}M_G)^{\frac{1}{1-b}}=\min f_{\log M_G};
\end{equation}
so when $M_G$ is large, we have $D_G\sim eM_G\log M_G$. 

\begin{example}\label{gl} Let $G=SL_n$. In this case it is not hard to check that 
$R_G=\frac{1}{n-1}$ (the ratio of the radii 
of the inscribed and circumscribed spheres for a regular simplex in $\Bbb R^{n-1}$). 
Also we have ${\rm h}=n$, so $M_G=(n-1)\sqrt{n+1}$. 
\end{example} 

We would like to provide uniform in $\lambda$ upper and lower bounds for 
$$
B=B(\lambda):=\frac{\pi^{\frac{r}{2}}}{\Gamma(\frac{r}{2}+1)}
DH_\lambda(0).
$$
Note that since $DH_\lambda$ is a probability measure supported in the unit ball, 
we have $B\ge 1$. However, the theorem below shows that 
in fact $B$ has to be much bigger. 

\begin{theorem}\label{Avalue} For all $\lambda$ we have  
$$
e^{-1}\frac{({\rm h}+1)^{\frac{r}{2}}}{\frac{r}{2}+1}\le B(\lambda)\le E_G, 
$$
where 
\begin{equation}\label{Egdef}
E_G:=10^5 r^{\frac{1}{\log M_G}}D_G^r,
\end{equation}
where $M_G$ is defined by \eqref{formmg} and $D_G$ is defined by \eqref{formdg}.
\end{theorem} 

\begin{proof}  First we establish the lower bound. By Lemma \ref{averasq}(iii), 
$$
B\ge ({\rm h}+1)^{\frac{r}{2}}T^{\frac{r}{2}}(1-T).
$$
This bound is optimized when $T=\frac{\frac{r}{2}}{\frac{r}{2}+1}$ and gives 
$$
B\ge ({\rm h}+1)^{\frac{r}{2}}\left(\frac{\frac{r}{2}}{\frac{r}{2}+1}\right)^{\frac{r}{2}}\frac{1}{\frac{r}{2}+1}\ge e^{-1}\frac{({\rm h}+1)^{\frac{r}{2}}}{\frac{r}{2}+1},
$$
as claimed. 

Now we prove the upper bound. 
Let $A:=A(0)=DH_\lambda(0)$.
Since $DH_\lambda$ is log-concave, for each $p\in \h_{\Bbb R}^*$ and $0\le t\le 1$,
we have 
$$
DH_\lambda(tp)\ge A^{1-t}DH_\lambda(p)^t.
$$
It follows that for $|p|^2=R_G^2L$, $0\le t\le 1$
$$
DH_\lambda(tp)\ge A^{1-t}A(L)^t.
$$
Thus 
$$
1\ge \int_{|p|\le R_G^2L}DH_\lambda(p)dp\ge A\frac{2(\pi R_G^2L)^{\frac{r}{2}}}{\Gamma(\frac{r}{2})}\int_0^1 t^{r-1}(\tfrac{A}{A(L)})^{-t}dt.
$$
So replacing the interval of integration by $[a,b]$ with $0<a\le b<1$ and estimating the 
factors of the integrand by boundary values, we get 
$$
A^{1-b}A(L)^b\frac{2(\pi R_G^2L)^{\frac{r}{2}}a^{r-1}}{\Gamma(\frac{r}{2})}(b-a)\le 1.
$$
Thus using Corollary \ref{supinfcor}(ii), we get 
$$
\left(\frac{rB}{2}\right)^{1-b}\left(\frac{1}{{\rm h}+1}-L\right)^b(R_G^2L)^{\frac{r}{2}}a^{r-1}(b-a)\le 1.
$$
This bound is optimized for fixed $b,L$ when $a=\frac{r-1}{r}b$, where the function $a^{r-1}(b-a)$ 
attains its maximum value $(\frac{r-1}{r})^{r-1}\frac{b^r}{r}$, so using that 
$(\frac{r-1}{r})^{r-1}\ge e^{-1}$, we get 
$$
\left(\frac{rB}{2}\right)^{1-b}\left(\frac{1}{{\rm h}+1}-L\right)^b(R_G^2L)^{\frac{r}{2}}\frac{b^r}{er}\le 1.
$$
This bound is optimized in $L$ for each $b$ when $L=\frac{1}{{\rm h}+1}\frac{r}{2b+r}$. 
Thus, using that $b\le 1$ and $(1+\frac{2}{r})^{\frac{r}{2}}\le e$, we get 
$$
\left(\frac{rB}{2}\right)^{1-b}\left(\frac{b}{(b+\frac{r}{2})({\rm h}+1)}\right)^b\left(\frac{b}{M_G}\right)^{r}\frac{1}{e^2r}\le 1.
$$
Since $b^b\ge e^{-\frac{1}{e}}$, we obtain  
$$
\left(\frac{rB}{2}\right)^{1-b}((b+\tfrac{r}{2})({\rm h}+1))^{-b}\left(\frac{b}{M_G}\right)^{r}\frac{1}{e^{2+\frac{1}{e}}r}\le 1.
$$
It follows using $r+1\le {\rm h}$ and $b\le 1$ that 
$$
\frac{rB}{2}
\le ({\rm h}+1)^{\frac{2b}{1-b}}\left(\frac{M_G}{b}\right)^{\frac{r}{1-b}}(e^{2+\frac{1}{e}}r)^{\frac{1}{1-b}}.
$$

Let $b_0:=b_0(\log M_G)$. Then 
$$
b_0<\frac{1}{1+\log M_G}\le \frac{1}{1+\log \sqrt{\rm h+1}}\le \frac{1}{1+\log \sqrt{3}}.
$$ 
So $\frac{b_0}{1-b_0}<\frac{1}{\log M_G}<\frac{1}{\log \sqrt{\rm h+1}}$ and $\frac{1}{1-b_0}\le 1+\frac{1}{\log M_G}<2.83$. 
Thus we have that ${(h+1)}^{\frac{2b_0}{1-b_0}}<e^4$. 
We then get, substituting $b=b_0$ and using  
that $2e^{4+2.83(2+\frac{1}{e})}<10^5$, that 
$$
B\le 10^5 r^{\frac{1}{\log M_G}}D_G^r,
$$
which completes the proof.
\end{proof} 

\begin{example} According to Theorem \ref{Avalue} and Example \ref{gl}, for $G=SL_n$ we have 
$$
2e^{-1}(n+1)^{\frac{n-3}{2}}\le \frac{\pi^{\frac{n-1}{2}}}{\Gamma(\frac{n+1}{2})}
DH_\lambda(0)\le O(n^{(\frac{3}{2}+\varepsilon)n})
$$
for all $\varepsilon>0$. 
\end{example} 

\begin{remark} 
For $\mu\in P_+$ let $\omega(\mu)$ be the minuscule weight 
such that $\mu-\omega(\mu)$ belongs to the root lattice $Q$ of $G$. Recall \cite{He} that 
$$
\lim_{N\to \infty}N^r  \frac{\dim L_{[N\lambda]}[\omega([N\lambda])]}{\dim L_{[N\lambda]}}=DH_\lambda(0).
$$
Thus the bounds of Theorem \ref{Avalue} provide estimates of maximal weight multiplicities of $L_\lambda$ for large $\lambda\in P_+\cap Q$. 
\end{remark}

\section{An upper bound for the asymptotic character} 

\subsection{An upper bound for the Fourier transform for a piecewise polynomial function} 

\begin{lemma}\label{1var} 
There exists a constant $C_N>0$ such that for any polynomial $Q\in \Bbb C[x]$ 
of degree $<N$ and a finite interval $I\subset \Bbb R$ we have 
$$
\bigg|\int_I Q(p)e^{ipx}dp \bigg|\le \frac{C_N}{|x|}\max_I|Q|.
$$
\end{lemma} 

\begin{proof} By shifting and rescaling we may assume without loss of generality that $I=[0,1]$. Let 
$$
F_n(x):=\int_0^1 p^n e^{ipx}dp=\frac{1}{x^{n+1}}\int_0^x p^ne^{ip}dp.
$$ 
It is easy to see that $\int_0^x p^ne^{ip}dp=P_n(x)e^{ix}-P_n(0)$  where $P_n$ is a polynomial of degree $n$, and  
$$
F_n(x)=\frac{P_n(x)e^{ix}-P_n(0)}{x^{n+1}}.
$$
Since $F_n$ is regular at $x=0$, $P_n(x)e^{ix}=P_n(0)+O(x^{n+1})$, $x\to 0$. Hence 
$P_n$ is a multiple of the $n$-th Taylor polynomial 
$$
T_n(x)=\sum_{j=0}^n \frac{(-ix)^j}{j!}
$$ 
of the function $e^{-ix}$ at $x=0$. 
Finally, since $F_n(0)=\frac{1}{n+1}$, we get 
$P_n=i^{-n-1}n! T_n$, i.e., 
 $$
F_n(x)=n! \frac{T_n(x)-e^{-ix}}{(-ix)^{n+1}}e^{ix}.
$$
So setting 
$$
M_n:=\sup_{x\in \Bbb R}\bigg|\frac{n!(T_n(x)-e^{-ix})}{x^n}\bigg|,
$$ 
we obtain 
$$
|F_n(x)|\le \frac{M_n}{|x|},\ x\in \Bbb R.
$$
Now if $Q(p)=\sum_{n=0}^N a_np^n$ then 
$$
\bigg|\int_0^1 Q(p)e^{ipx}dp \bigg|=\bigg|\int_0^1 \sum_{n=0}^{N-1} a_np^ne^{ipx}dp \bigg|\le 
\sum_{n=0}^{N-1} \bigg|\int_0^1 a_np^ne^{ipx}dp \bigg|\le \frac{\sum_{n=0}^{N-1} M_n|a_n|}{|x|}.
$$
Clearly, there exists $C_N$ such that 
$$
\sum_{n=0}^{N-1} M_n|a_n|\le C_N\max_I |Q|
$$ 
(as both sides define norms on the $N$-dimensional space of polynomials in degree $<N$). This proves the lemma. 
\end{proof} 

Lemma \ref{1var} immediately implies 

\begin{corollary}\label{1var1} Suppose $F: I\to \Bbb C$ is a piecewise polynomial function of degree $<N$ on a finite interval $I\subset \Bbb R$, with $m$ polynomiality intervals. 
Then 
$$
\bigg|\int_I F(p)e^{ipx}dp \bigg| \le \frac{C_{N}m}{|x|}\sup |F|.
$$
\end{corollary} 

\subsection{Proof of Theorem \ref{upbou}} 

We have 
$$
\int_{\Bbb \h_{\Bbb R}^*} DH_\lambda(p)e^{ipx}dp=\int_{\Bbb R}F_x(u)e^{iu|x|}du,
$$
where 
$$
F_x(u):=\int_{p:\ px=u|x|}DH_\lambda(p)dp.
$$
It is clear that $F_x$ is piecewise polynomial of degree $<|R_+|$ and it has 
$\le |W|-1$ polynomiality intervals on its support. Also since the support of $DH_\lambda$ 
is contained in the unit ball, by Theorem \ref{Avalue} we have 
$$
\sup|F_x|\le \frac{\Gamma(\frac{r}{2}+1)}{\sqrt{\pi}\Gamma(\frac{r+1}{2})}E_G.
$$
Therefore Theorem \ref{upbou} follows from Corollary \ref{1var1} with 
\begin{equation}\label{formcg}
C(G)= \frac{\Gamma(\frac{r}{2}+1)}{\sqrt{\pi}\Gamma(\frac{r+1}{2})}C_{|R_+|}E_G (|W|-1).
\end{equation} 
where $C_{|R_+|}$ is $C_N$ for $N=|R_+|$ and $E_G$ is defined by \eqref{Egdef}.

\subsection{The true rate of decay of the asymptotic character} \label{true}

For any $G$ define 
$$
\gamma(G):=\inf\lbrace \gamma: |X(\lambda,x)|=O(|\lambda|^{-\gamma} |x|^{-\gamma}).
$$
This number characterizes the actual (slowest) rate of decay of the asymptotic character at infinity. Theorem \ref{upbou} shows that $\gamma(G)\ge 1$. 
This bound is sharp for $G=SL_n$: 
when $\lambda=\omega_1,x=t\omega_1^\vee$ then it is easy to compute that 
$$
X(\lambda,x)=(n-1)e^{-\frac{it}{n}}\int_0^1 p^{n-2}e^{ipt}dp=
(n-1)! \frac{T_{n-2}(t)-e^{-it}}{(-it)^{n-1}}e^{i\frac{n-1}{n}t}, 
$$
which decays as $t^{-1}$ at infinity (as $\deg T_{n-2}=n-2$), 
hence $\gamma(SL_n)=1$. On the other hand, for other groups this bound is not sharp. For example, 
it is not difficult to check using \eqref{sp4} that for $G=Sp_4$
one has $|X(\lambda,x)|=O(|\lambda|^{-2}|x|^{-2})$, while 
for $\lambda=\omega_1$ (highest weight of the vector representation) 
and $x=t\omega_1^\vee$ we have 
$$
X(\lambda,x)=6\frac{t-\sin t}{t^3},
$$
so $\gamma(Sp_4)=2$. More generally, a similar computation shows that for the Coxeter generalization $\gamma(I_{2m})=m-2$, in particular $\gamma(G_2)=4$.  

In general, since $DH_\lambda$ is piecewise polynomial, 
$\gamma(G)$ is an integer. Moreover, we expect that one can use the theory of \cite{GLS}  to show that $\gamma(G)$ is the minimum over $\lambda\ne 0$, $x\ne 0$ of the number of positive roots $\alpha$ such that $(\alpha,x)\ne 0$, $(\alpha,\lambda)\ne 0$, which we denote by $\mu(G)$. It is clear that $\mu(G)$ is the minimum, over $1\le i,j\le r$ and weights 
$\mu\in W\omega_i,\nu\in W\omega_j$, of the number of positive roots 
$\alpha$ such that $(\alpha,\mu)\ne 0$, $(\alpha,\nu)\ne 0$. In other words, 
$\mu(G)$ is the smallest size of a subset of positive roots of $G$ such that the rest of positive roots are contained in the union of two hyperplanes. 

It is easy to show that $\mu(A_n)=1$ (see below), but for other cases $\mu(G)>1$, as shown by the following proposition. 

\begin{proposition}\label{SGer} (\cite{SG}) One has 
$\mu(A_{n-1})=1$, $\mu(B_n)=\mu(C_n)=\mu(D_n)=2$, 
$\mu(I_{2m})=m-2$, $\mu(H_3)=6$, $\mu(H_4)=32$, 
$\mu(F_4)=8$, $\mu(E_6)=6$, $\mu(E_7)=11$, 
$\mu(E_8)=24$.  
\end{proposition} 

\begin{proof}  First of all, the root system $A_{n-1}$ cannot be covered by two hyperplanes. 
This is easy to prove by induction in $n$. Indeed, if it is covered by $E_1,E_2$ 
then by the induction assumption the two standard $A_{n-2}$ subsystems 
should be entirely contained in $E_1$ and $E_2$, hence span them. But then 
$e_1-e_n$ is not contained in either of these hyperplanes, a contradiction. 
This implies that $\mu(A_{n-1})=1$, since these hyperplanes (defined by $x_1=0$ and $x_n=0$) 
contain all positive roots but $e_1-e_n$. 

Similarly, $\mu(D_n),\mu(B_n),\mu(C_n)\le 2$
since these hyperplanes contain all positive roots of these systems but $e_1-e_n$ 
and $e_1+e_n$. On the other hand, it is easy to check directly that $\mu(D_4)=2$. 
So if we have two hyperplanes $E_1$ and $E_2$ and $\alpha$ is a positive root of $D_n$ 
not contained in $E_1\cup E_2$ (which exists by the $A_{n-1}$ case), then fixing a $D_4$ subsystem containing $\alpha$ and letting $H$
be the span of this subsystem, we get that $(H\cap E_1)\cup (H\cap E_2)$ misses at least 
two roots $\alpha,\beta$, which are therefore missed by $E_1\cup E_2$. This shows that 
$\mu(D_n)=\mu(B_n)=\mu(C_n)=2$. 

Also, as shown above, $\mu(I_{2m})=m-2$. The $H_3$ case can be checked by looking at the regular dodecahedron. Finally, the cases $H_4,F_4,E_6,E_7,E_8$ 
have been checked by Antoine de Saint Germain using the programs SageMath and GAP (in principle, this can also be done by hand, but is tedious).  
\end{proof} 

Finally, note that the invariant $\gamma(G)$ also controls the smoothness degree of the multivariate spline function $DH_\lambda$. Namely, for all $\lambda\in \h^*_{\Bbb R}$, $DH_\lambda$ admits $\gamma(G)-1$ piecewise polynomial derivatives (as a distribution). In particular, if $\gamma(G)\ge 2$ then $DH_\lambda$ belongs to $C^{\gamma(G)-2}(\h_{\Bbb R})$. 

\section{Proof of Corollary \ref{GGRconj}}

As pointed out in \cite{GGR}, we have 
$$
\frac{\chi_\lambda(e^{ix})}{\chi_\lambda(1)}=\frac{X(\lambda+\rho,x)}{X(\rho,x)}.
$$
By Corollary \ref{coro1}, there exists $x_0\in \h_{\Bbb R}$ with $|x_0|\le \frac{C(G)}{c(G)|\lambda+\rho|}$
such that 
$$
{\rm Re}X(\lambda+\rho,x_0)=-c(G,\lambda)\le -c(G).
$$ 
By Corollary \ref{co1}, 
if $|\lambda+\rho|< \frac{C(G)|\rho|}{c(G)\sqrt{2d}}$ then $0<X(\rho,x_0)<1$. 
It follows that in this case 
$$
{\rm Re}\frac{\chi_\lambda(e^x)}{\chi_\lambda(1)}\le -c(G),
$$
as desired. 

\begin{remark} Note that for the proof of Theorem \ref{upbou} and Corollary \ref{GGRconj} we only need the {\bf existence} of the bound $E_G$ in Theorem \ref{Avalue} (rather than an explicit formula for it), which is immediate since $DH_\lambda(0)$ is continuous in $\lambda$. Thus Theorem \ref{Avalue} is not really needed for the proofs of these statements. 
\end{remark}

\section{An upper bound for $c(SL_n)$} 

\begin{proposition} 
We have $c(SL_n)\le (\frac{4}{\pi^2})^{n-2}$. 
\end{proposition} 

\begin{proof} We will show that 
$c(SL_n,\rho)\le (\frac{4}{\pi^2})^{n-2}$, which is sufficient. 
Recall that 
$$
X(\rho,2x)=\prod_{\alpha\in R_+}\frac{\sin (\alpha,x)}{(\alpha,x)}=\prod_{i<j}\frac{\sin(x_i-x_j)}{x_i-x_j},
$$
where $\mathfrak h\cong \mathfrak h^*$ is realized as 
the hyperplane in $\Bbb C^n$ defined by the equation 
$\sum_{j=1}^nx_j=0$. 
Our job is to show that if this product is negative, then it must be $\ge - (\frac{4}{\pi^2})^{n-2}$. 
To this end, assume that $x_1\ge...\ge x_n$ and $X(\rho,2x)<0$. Then at least one of the factors must be negative, so we must have $L:=x_1-x_n\ge \pi$. Also, we have 
$$
|X(\rho,2x)|\le \bigg|\prod_{j=2}^{n-1} \frac{\sin(x_1-x_j)\sin(x_j-x_n)}{(x_1-x_j)(x_j-x_n)}\bigg|.
$$
Let 
$$
K:=\sup_{a+b\ge \pi,a,b\ge 0}|\tfrac{\sin(a)\sin(b)}{ab}|.
$$
It follows that $|X(\rho,2x)|\le K^{n-2}$, and thus it remains to show that $K\le \frac{4}{\pi^2}$.  Note that
\[
|\tfrac{\sin(a)\sin(b)}{ab}|
\le
|\tfrac{\sin(a)}{a}|
\le
\tfrac{1}{|a|},
\]
and thus we may restrict the supremum to $a,b\in [0,\frac{4}{\pi^2}]$, or for simplicity to $a,b\in [0,\pi]$.

Note that
$$
\frac{d}{da}\frac{\sin a}{a}=\frac{a\cos a-\sin a}{a^2}=\frac{(a-\tan a)\cos a}{a^2}
$$
 is negative on $(0,\pi)$. Thus $\frac{\sin a}{a}$ is decreasing on $[0,\pi]$.
So 
$$
K=\sup_{a\in [0,\pi]}\sup_{a\le L\le \pi}|\tfrac{\sin(a)\sin(L-a)}{a(L-a)}|=\sup_{a\in [0,\pi]}\frac{\sin a\sin (\pi-a)}{a(\pi-a)}=\sup_{a\in [0,\pi]}\frac{\sin^2 a}{a(\pi-a)}.
$$
The function $g(a):=\frac{\sin^2 a}{a(\pi-a)}$ satisfies $g(a)=g(\pi-a)$ and is log-concave on $(0,\pi)$ since 
$$
\frac{d^2}{da^2}\log g(a)=-\frac{2}{\sin^2 a}+\frac{1}{a^2}+\frac{1}{(\pi-a)^2}<0.
$$
Thus it must attain its maximum at $a=\frac{\pi}{2}$, so $K=g(\frac{\pi}{2})=\frac{4}{\pi^2}$, as desired. 
\end{proof} 

\section{Appendix: Mittag-Leffler type sums associated with root systems}

\subsection{The main theorem}

We keep the notation of the body of the paper, and let $P,P^\vee$ and $Q,Q^\vee$ 
denote the (co)weight and (co)root lattices of $G$. Let $k$ be a positive integer. Consider the function on $\h_{\Bbb R}$ given by 
$$
f(x):=X(\rho,2\pi x)=\prod_{\alpha\in R_+}\frac{\sin\pi (\alpha,x)}{\pi(\alpha,x)}.
$$

Let $Z\subset G$ be the center and $\xi: Z\to \Bbb C^*$ a character. Since $Z=P^\vee/Q^\vee$, we may view $\xi$ as a character of $P^\vee$ which is trivial on $Q^\vee$. Define the function 
$$
F_{k,\xi}(x):=\sum_{a\in P^\vee}\xi(a)f^k(x+a).
$$
(If $G=SL_2$ and $k=1$ then the sum is not absolutely convergent, and should be understood in the sense of principal value). Thus, the meromorphic function 
$$
M_{k,\xi}(x):={\frac{F_{k,\xi}(x)}{\prod_{\alpha\in R_+}\pi^{-k}\sin^k \pi(\alpha,x)}}
$$ 
has a Mittag-Leffler type decomposition 
$$
M_{k,\xi}(x)=\sum_{a\in P^\vee}\frac{(-1)^{k(2\rho,a)}\xi(a)}{\prod_{\alpha\in R_+}(\alpha,x)^k}.
$$
For example, for $G=SL_2$, we have 
$$
F_{1,1}(x)=1, F_{1,-1}(x)=\cos \pi x, 
$$
which gives the classical Mittag-Leffler decompositions 
$$
\pi\cot \pi x=\sum_{n\in \Bbb Z}\frac{1}{x+n},\ \frac{\pi}{\sin \pi x}=\sum_{n\in \Bbb Z}\frac{(-1)^n}{x+n}. 
$$

The goal of this note is to prove the following theorem. 

\begin{theorem}\label{mainthe} 
The function $F_{k,\xi}$ is a $W$-invariant trigonometric polynomial on the maximal torus $T:=\h/Q^\vee$
of $G$, which is a nonnegative rational linear combination of irreducible characters of $G$. 
\end{theorem} 

For $G=SL_n$ and $\xi$ being a character of order $2$, this theorem was conjectured by R. Coquereaux and J.-B. Zuber (\cite{CZ}, Conjecture 1 in Subsection 2.2). Moreover, this paper provides an explicit decomposition of $F_{k,\xi}$ for such $\xi$ into characters of $SL_n$ for $n\le 6$, while such
decompositions for $SL_7$ and $G_2$ are computed in \cite{C} and for 
${\rm Spin}_5$ and ${\rm Spin}_7$ in \cite{CMZ}. Also \cite{C} gives the list of characters that occur in the decomposition of $F_{k,\xi}$ for several other
Lie groups.

\subsection{Proof of the main theorem} 

\subsubsection {Contraction of representations} 
We start with the following general fact.  

\begin{proposition}\label{p1}(\cite{Li1}, Proposition 3) Let $(V,\rho_V)$ be a rational representation of $G$, and $N$ a positive integer. Let $V_N$ be the direct sum of all the weight subspaces of $V$ of weights divisible by $N$. Then the action of $T$ on $V_N$ given by $t\circ v:=\rho_V(t^{\frac{1}{N}})v$ extends to an action of $G$.\footnote{Note that $\rho_V(t^{\frac{1}{N}})v$ is independent on the choice of the $N$-th root $t^{\frac{1}{N}}$.} In other words,  
$$
\chi_{V,N}:=\sum_{\lambda\in P}\dim V[N\lambda]e^\lambda
$$
is a nonnegative linear combination of irreducible characters of $G$. Namely, the multiplicity of 
$\chi_\lambda$ in $\chi_{V,N}$ equals the multiplicity of 
$L_{N\lambda+(N-1)\rho}$ in $V\otimes L_{(N-1)\rho}$. 
\end{proposition} 

\begin{proof} Littelmann proves this proposition via his path model as an illustration of its use, 
but we give a more classical proof using the Weyl character formula. We have to show that the integral
$$
I:=\int_{\h_\Bbb R/Q^\vee}\sum_{\lambda\in P}\dim V[N\lambda]e^{-2\pi i\lambda(x)}\chi_\lambda(e^{2\pi ix})|\Delta(x)|^2dx
$$
is nonnegative, where $\Delta(x)$ is the Weyl denominator, since the multiplicity in question is $I/|W|$. 

Denoting the character of $V$ by $\chi_V$, we have 
$$
I=\int_{\h_\Bbb R/Q^\vee}\int_{\h_\Bbb R/Q^\vee}\sum_{\lambda\in P}\overline{\chi_V(e^{2\pi iy})}e^{2\pi iN(\lambda,y)}e^{-2\pi i(\lambda,x)}\chi_\lambda(e^{2\pi ix})|\Delta(x)|^2dydx=
$$
$$
\int_{\h_\Bbb R/Q^\vee}\int_{\h_\Bbb R/Q^\vee}\overline{\chi_V(e^{2\pi iy})}\delta(x-Ny)\chi_\lambda(e^{2\pi ix})|\Delta(x)|^2dydx=
$$
$$
\int_{\h_\Bbb R/Q^\vee}\overline{\chi_V(e^{2\pi iy})}\chi_\lambda(e^{2\pi iNy})|\Delta(Ny)|^2dy.
$$
Using the Weyl character formula, we then have 
$$
I=\int_{\h_\Bbb R/Q^\vee}\overline{\chi_V(e^{2\pi iy})}\left(\sum_{w\in W}(-1)^we^{2\pi i(w(\lambda+\rho),Ny)}\right)\overline{\Delta(Ny)}dy=
$$
$$
\int_{\h_\Bbb R/Q^\vee}\overline{\chi_V(e^{2\pi iy})}\frac{\sum_{w\in W}(-1)^we^{2\pi i(w(\lambda+\rho),Ny)}}{\Delta(y)}\frac{\overline{\Delta(Ny)}}{\overline{\Delta(y)}}|\Delta(y)|^2dy=
$$
$$
\int_{\h_\Bbb R/Q^\vee}\overline{\chi_V(e^{2\pi iy})}\chi_{N\lambda+(N-1)\rho}(e^{2\pi iy})\frac{\overline{\Delta(Ny)}}{\overline{\Delta(y)}}|\Delta(y)|^2dy.
$$
Now recall that $\frac{\Delta(Ny)}{\Delta(y)}=\chi_{(N-1)\rho}(e^{2\pi iy})$. Thus we get 
$$
I=\int_{\h_\Bbb R/Q^\vee}\overline{\chi_V(e^{2\pi iy})\chi_{(N-1)\rho}(e^{2\pi iy})}\chi_{N\lambda+(N-1)\rho}(e^{2\pi iy})|\Delta(y)|^2dy,
$$
i.e., $I/|W|$ is the multiplicity of $L_{N\lambda+(N-1)\rho}$ in $V\otimes L_{(N-1)\rho}$, as desired. 
\end{proof} 

\begin{remark} 1. If $N$ is odd (and coprime to $3$ for $G$ of type $G_2$), then 
Proposition \ref{p1} has a nice representation-theoretic interpretation. Namely, 
if $q$ is a root of unity of order $N$ and $G_q$ the corresponding Lusztig quantum group, 
then there is an exact {\it contraction functor} $F: {\rm Rep} G_q\to {\rm Rep} G$ which 
at the level of $P$-graded vector spaces transforms $V$ into $V_N$ with weights divided by $N$ 
(see \cite{GK} and references therein). Proposition \ref{p1} is then obtained by applying 
the functor $F$ to a Weyl module. 

2. Suppose that $G$ is not simply laced. Normalize the inner product on $\h^*$ so that long roots 
have squared length $2$. This inner product identifies $\h$ with $\h^*$ so that $\alpha_i^\vee$ map to $2\alpha_i/(\alpha_i,\alpha_i)$. Note that $2/(\alpha_i,\alpha_i)$ is an integer, so under this identification $Q^\vee\subset Q$, hence $P^\vee\subset P$. Let $V_N'\subset V_N$ be the span of the weight subspaces of $V$ of weights belonging to $NP^\vee$ with weights divided by $N$. Then, analogously to Proposition \ref{p1}, $V_N'$ extends to a representation of the Langlands dual Lie algebra $\g^L$, with a similar descrition of multiplicities (\cite{Li1}, Proposition 4). Note that this statement is nontrivial even if $\g^L\cong \g$, since the arrow on the Dynkin diagram is reversed. This also has a representation-theoretic 
interpretation similar to (1), see \cite{Li1}, Section 3, \cite{GK}.

3. As explained in \cite{Li1}, 
Proposition \ref{p1} generalizes to symmetrizable Kac-Moody algebras (both our proof and that of \cite{Li1} can 
be straightforwardly extended to this case). So does the non-simply laced version of Proposition \ref{p1} given in (2) and the above representation-theoretic interpretations, see \cite{Li2}.
\end{remark}

\subsubsection{Proof of Theorem \ref{mainthe}} 

For simplicity assume that $\lambda$ is regular and $G\ne SL_2$. Then $DH_\lambda$ is absolutely continuous with respect to the Lebesgue measure (i.e., the density function $DH_\lambda(p)$ is continuous), and it is known (\cite{GLS,He}) that if $\mu_N\in P,\lambda_N\in P_+$ are sequences such that $\frac{\mu_N}{N}\to \mu, \frac{\lambda_N}{N}\to \lambda$ as $N\to \infty$ and $\lambda_N-\mu_N\in Q$ then 
\begin{equation}
\lim_{N\to \infty}N^r \frac{\dim L_{\lambda_N}[\mu_N]}{\dim L_{\lambda_N}}=DH_\lambda(\mu).
\end{equation}

\begin{proposition}\label{p2} Let $\lambda_1,...,\lambda_k\in \h_{\Bbb R}^*$ be regular dominant weights. Then the trigonometric polynomial
$$
\sum_{\mu\in P}(DH_{\lambda_1}*...*DH_{\lambda_k})(\mu)e^\mu
$$
(where $*$ denotes convolution of measures) is a linear combination of irreducible characters of $G$ with nonnegative real coefficients. 
\end{proposition} 

\begin{proof}  First assume that $\lambda_i$ are rational, and let $D$ be their common denominator. 
Then, taking the limit as $N\to \infty$ in Proposition \ref{p1} with 
$V=L_{N\lambda_1}\otimes...\otimes L_{N\lambda_k}$ and $N$ divisible by $D$, we
obtain the desired statement. Now the general case follows from the facts that rational weights are dense 
in $\h_{\Bbb R}$ and $DH_\lambda(\mu)$ is continuous in $\lambda$. 
\end{proof} 

Since the asymptotic character is the Fourier transform of the Duistermaat-Heckman measure, Theorem \ref{mainthe} follows from Proposition \ref{p2} by taking 
$\lambda_1,...,\lambda_k=\rho$. The rationality of the coefficients follows from the rationality of the values of the convolution power $(DH_\rho)^{*k}$ at rational points. 

\subsubsection{The characters occuring in $F_{k,\xi}$.} 

Let us now discuss which irreducible characters can occur in the decomposition of $F_{k,\xi}$. 
Let us view $\xi$ as an element of $P/Q$. 
Clearly, if $\chi_\lambda$ occurs in $F_{k,\xi}$ then the central character of the representation 
$L_\lambda$ must be $\xi=k\rho-\lambda$ mod $Q$. 
If so, then, as shown above, the multiplicity of $\chi_\lambda$ in $F_{k,\xi}$ 
is $(DH_\rho)^{*k}(\lambda)$. Since this density is continuous and supported on $\Pi_{k\rho}$ (the convex hull of the Weyl group orbit of $k\rho$), we see that if $\chi_\lambda$ occurs then $\lambda$ has to be strictly in the interior of $\Pi_{k\rho}$. 

Let $m_i(\xi)$ be the smallest strictly positive number such 
that $m_i(\xi)=(\xi,\omega_i^\vee)$ in $\Bbb R/\Bbb Z$, and let 
$\beta_\xi:=\sum_i m_i(\xi)\alpha_i\in P$. 
Then we get 

\begin{proposition}\label{p3} The character $\chi_\lambda$ occurs in $F_{k,\xi}$  if and only if 
$\xi=k\rho-\lambda$ mod $Q$ and 
$(\lambda,\omega_i^\vee)<k(\rho,\omega_i^\vee)$ for all $i$. Moreover, in  presence of the first condition, 
the second condition is equivalent to the inequality $\lambda\le k\rho-\beta_\xi$. 
\end{proposition} 

We also have 

\begin{proposition} The weight $\rho-\beta_\xi$ (and hence $k\rho-\beta_\xi$ for all $k\ge 1$) is dominant. 
\end{proposition} 

\begin{proof} We need to show that for all $i$ we have $(\rho-\beta_\xi,\alpha_i^\vee)\ge 0$. 
Since $\rho-\beta_\xi$ is integral, it suffices to show that $(\rho-\beta_\xi,\alpha_i^\vee)>-1$, i.e., 
$(\beta_\xi,\alpha_i^\vee)<2$. But $(\beta_\xi,\alpha_i^\vee)=2m_i+\sum_{j\ne i}m_j(\alpha_j,\alpha_i^\vee)<2$, since $0<m_j\le 1$ and $(\alpha_j,\alpha_i^\vee)\le 0$ for all $j\ne i$ and is strictly negative for some $j$. 
\end{proof} 

\begin{corollary} We have 
$$
F_{k,\xi}=\sum_{\mu\le k\rho-\beta_\xi}C_{k,\xi}(\mu)\chi_\mu,
$$
where $C_{k,\xi}(\mu)\in \Bbb Q_{>0}$. In particular, the leading term is a multiple of $\chi_{k\rho-\beta_\xi}$. 
\end{corollary}

\end{document}